\documentclass[preprint]{jmd}
\usepackage{amssymb}
\usepackage{graphicx}
\UseLinks
\microtypecontext{spacing=nonfrench}
\theoremstyle{plain}
\newtheorem{theorem}{Theorem}%[section]
\newtheorem{lemma}[theorem]{Lemma}
\newtheorem{proposition}[theorem]{Proposition}

\newtheorem{corollary}[theorem]{Corollary}

\theoremstyle{definition}
\newtheorem{definition}[theorem]{Definition}
\newtheorem{example}[theorem]{Example}
\newtheorem{remark}[theorem]{Remark}
\newtheorem*{question}{Problem}
\numberwithin{theorem}{section}

\def\etc.{\emph{et\thinspace c.}}
\def\Wloc{W_{\text{loc}}}
\DeclareMathOperator{\Id}{Id}

\renewcommand{\d}{\,\mathrm{d}}

\def\R{\mathbb{R}}
\def\Q{\mathbb{Q}}

\def\N{\mathbb{N}}
\def\T{\mathbb{T}}
\def\H{\mathbb{H}}
\def\Z{\mathbb{Z}}

\def\SS{\mathbb{S}}

\def\Dc{\mathcal{D}}

\let\emptyset\varnothing

\let\rest\restriction
\newcommand{\aeq}{\stackrel{\scriptscriptstyle\textrm{\normalfont a.e.\!}}{=}}
\newcommand{\aeto}{\xrightarrow{\scriptscriptstyle\textrm{\normalfont a.e.\!}}}

\newcommand{\st}{\mathbin{\text{$\phantom{|}$\vrule width.04em$\phantom{|}$}}}
\newcommand{\const}{\operatorname{const.}}%
\begin{document}
\title{Multiple mixing from weak hyperbolicity\breakhere by the Hopf argument}
\author{Yves Coud\`ene}
\email{yves.coudene@univ-brest.fr}
\address{Universit\'e de Bretagne Occidentale, 6 av.\ Victor Le Gorgeu, 29238 Brest CEDEX, France}
\author{Boris Hasselblatt}
\email{Boris.Hasselblatt@tufts.edu}
\address{Office of the Provost and Senior Vice
  President\\
Tufts University\\
Medford, MA 02155\\
USA}
\curraddr{Aix Marseille Universit\'e, CNRS, Centrale Marseille, Soci\'et\'e Math\'ematique de France, Centre International de Rencontres Math\'ematiques, UMS 822, and Institut de Math\'ematiques de Marseille, UMR 7373, 13453 Marseille, France}
\thanks{This work was partially supported by the
Committee on Faculty Research Awards of Tufts University and the Jean-Morlet Chair.}
\makeatletter
\def\st@artcaddr{\begingroup\punctuation\par\nobreak\indent{\itshape Address}\/:\space}
\makeatother
\author{Serge Troubetzkoy}
\email{serge.troubetzkoy@univ-amu.fr}
\address{Aix Marseille Universit\'e, CNRS, Centrale Marseille, I2M, UMR
  7373, 13453 Marseille, France}
\curraddr{I2M, Luminy, Case 907, F-13288 Marseille CEDEX 9, France}
%\curraddr{Aix Marseille Universit\'e, CNRS, Centrale Marseille, and Institut de Math\'ematiques de Marseille, UMR 7373, 13453 Marseille, France}
\thanks{This work was partially supported by ANR Perturbations}
\Subjclass{37D20}{57N10}
\begin{abstract}
We show that using only weak hyperbolicity (no smoothness, compactness or
exponential rates) the Hopf argument produces multiple mixing in an
elementary way. While this recovers classical results with far simpler
proofs, the point is the broader applicability implied by the weak
hypotheses. Some of the results can also be viewed as establishing ``mixing
implies multiple mixing'' outside the classical hyperbolic context.
\end{abstract}
\maketitle%\tableofcontents
\section{Introduction}
The origins of hyperbolic dynamical systems are connected with the efforts
by Boltzmann and Maxwell to lay a foundation under statistical mechanics.
In today's terms their fundamental postulate was that the mechanical system
defined by molecules in a container is ergodic, and the difficulties of
establishing this led to the search for \emph{any} mechanical systems with
this property. The motion of a single free particle (also known as the
geodesic flow) in a negatively curved space emerged as the first and for a
long time sole class of examples with this property. Even here,
establishing ergodicity was subtle enough that initially this was only done
for constantly curved surfaces by using the underlying algebraic structure.
Eberhard Hopf was the first to go beyond this context, and his argument
remains the main tool for deriving ergodicity from hyperbolicity in the
absence of an algebraic structure (the alternative tool being the theory of
equilibrium states). Our purpose is to show how much more than ergodicity
it can produce. Specifically, in its original form the Hopf argument
establishes ergodicity when the contracting and expanding partitions of a
dynamical system are jointly ergodic. We present a recent refinement
originally due to Babillot that directly obtains mixing from joint
ergodicity of these two partitions. Further, we publicize the observation
that the argument produces multiple mixing if the stable partition is
ergodic by itself, and we give a simple proof of ergodicity of the stable
foliation. Taken together, this gives a simple, self-contained general
proof of multiple mixing of which \Ref{CORAbstractMMix} is a prototype.

Here is how the results in this paper can be applied together. Use the Hopf
argument (\Ref{SApplications}) to establish mixing (or just total
ergodicity), deduce that the stable partition is ergodic
(\Ref{SThouvenot}), then apply the one-sided Hopf argument
(\Ref{SHopfArgument}) to obtain multiple mixing.
%While it is
%useful to have these steps at disposal individually, we can summarize them
%in a single theorem.
%\begin{theorem}\Label{THMNMixingCriterion}
%Consider a metric space $X$, a Borel probability measure $\mu$, and an
%invertible $\mu$-preserving transformation $f\colon X\to X$ with $W^{ss}$
%absolutely continuous and $W^{ss}$, $W^{su}$ defining a local product structure.
%If
%\[
%\varphi\in L^2(\mu)\ f\text{-invariant},
%\ W^{ss}\text{-saturated and }W^{su}\text{-saturated}\;\Rightarrow\;\varphi\aeq\const,
%\]
%then $f$ is ergodic. If
%\[
%\varphi\in L^2(\mu)%\ f\text{-invariant},
%\ W^{ss}\text{-saturated and }W^{su}\text{-saturated}\;\Rightarrow\;\varphi\aeq\const,
%\]
%then $f$ is multiply mixing.
%\end{theorem}
%\begin{theorem}\Label{THMNMixingCriterionFlows}
%Consider a metric space $X$, a Borel probability measure $\mu$, and a
%$\mu$-preserving flow $\varphi^t\colon X\to X$\JMDCOMMENT{Edit once conclusions stabilize} with $W^s$
%absolutely continuous and $W^s$, $W^u$ defining a local product structure.
%If
%\[
%f\in L^2(\mu)\ \varphi^t\text{-invariant},
%\ W^s\text{-saturated and }W^u\text{-saturated}\;\Rightarrow\;f\aeq\const,
%\]
%then $\varphi^t$ is ergodic. If
%\[
%f\in L^2(\mu)
%\ W^s\text{-saturated and }W^u\text{-saturated}\;\Rightarrow\;f\aeq\const,
%\]
%then $\varphi^t$ is mixing. If furthermore $\varphi^t$ is continuous, $W^s$
%is minimal with bounded Jacobians and a uniform product structure
%(\Ref{CORflowsmultmix}), then $\varphi^t$ is multiply
%mixing.
%\end{theorem}
We remark that our proofs are self-contained, quite short and do not use
compactness, differentiability, or exponential behavior; nor are the $W^i$
assumed to consist of manifolds. The step from ergodicity to multiple
mixing does not need the full force of the usual notions of local product
structure and absolute continuity. Indeed, in our applications to
billiards (\Ref{CORBilliardsMMixing}) and partially hyperbolic
dynamical systems (\Ref{THMBurnsWilkinson}), more information is available
than needed for our results.\looseness-1

%This applies to volume-preserving Anosov diffeomorphisms
%\cite[Chapter 6]{BrinStuck}:
%\begin{corollary}\Label{CORAnosovMultiMixing}
%Volume-preserving Anosov diffeomorphisms are multiply mixing.
%\end{corollary}
Hyperbolic dynamical systems on compact spaces enjoy even stronger stochastic properties,
such as the Kolmogorov property and being measurably isomorphic to a
Bernoulli system \cite[Theorem 4.1]{Bowen}. Our purpose is to show how much
follows from just the Hopf argument.

We conclude this introduction with the Hopf argument for ergodicity.
Consider a metric space $X$ with a Borel probability measure $\mu$ and a
$\mu$-preserving transformation $f\colon X\to X$. The stable partition
of\/ $f$ is defined by
\begin{equation}\label{eqDefStableSet}
%\begin{aligned}
W^{ss}(x)\dfn\{y\in X\st d(f^n(x),f^n(y))\lto{n\to+\infty}0\}
%\end{aligned}
\end{equation} 
\begin{definition}\Label{DEFWsSubordinate}
$\varphi\colon X\to\R$ is \emph{subordinate} to $W^{ss}$ or \emph{$W^{ss}$-saturated} if there is a set $G\subset X$ with $\mu(G)=1$ such that
$x,y\in G$ and $y\in W^{ss}(x)$ imply $\varphi(x)=\varphi(y)$. 
\end{definition}
\begin{remark}%\Label{REM}
In this case $\varphi^s(x)\dfn\begin{cases}0&\text{if }W^{ss}(x)\cap
G=\emptyset\\\varphi(y)&\text{if }y\in G\cap W^{ss}(x)\end{cases}\aeq\varphi$ is
(everywhere!) constant on stable sets.
\end{remark}
\begin{theorem}[Hopf Argument]\Label{THMHopfCoudene}\index{Hopf!argument}
If\/ $(X,\mu)$ is a metric Borel probability space, $f\colon X\to X$
$\mu$-preserving, then any $f$-invariant $\varphi\in L^p(\mu)$ is
$W^{ss}$-saturated.
\end{theorem}\def\CHIfFk{\tau_{F_k}}
\begin{proof}
%For $\varphi\in L^p(\mu)$ and $k\in\N$ 
The Luzin Theorem gives $F_k\subset X$
with $\mu(X\smallsetminus
F_k)<2^{-k}$ and $\varphi\rest{F_k}$ uniformly continuous. If $E_k\dfn\big\{x\in
X\st\frac12<\CHIfFk\dfn\lim\limits_{N\to\infty}\frac1N\#\{0\le
n<N\mid f^n(x)\in{F_k}\}\big\}$, then, using the Birkhoff ergodic theorem,
\[
\mu(X\smallsetminus E_k)=
2\int_{X\smallsetminus E_k}\hskip-.7em\nicefrac12
\le
2\int_{X\smallsetminus E_k}1-\CHIfFk
\le
2\int1-\CHIfFk
=
2\int\chi_{X\smallsetminus F_k}
=2\mu(X\smallsetminus F_k)<2^{1-k}\llap,
\]
while for $x,y\in E_k$ there are $n_i\lto{i\to\infty}\infty$ with
$\{f^{n_i}(x),f^{n_i}(y)\}\subset F_k$, since each has density $> 1/2$. If furthermore $y\in W^{ss}(x)$ and $\varphi$ is $f$-invariant, then
$\varphi(x)-\varphi(y)=\varphi(f^{n_i}(x))-\varphi(f^{n_i}(y))\lto{i\to\infty}0$, which
proves the claim on %$A_n\dfn\bigcap_{k\ge n}E_k$ and hence on
$\bigcup_{n\in\N}\bigcap_{k\ge n}E_k\aeq X$.
\end{proof}
If\/ $f$ is invertible, then we define
\[%\begin{equation}%\label{eqDefStableSet}
%\begin{aligned}
W^{su}(x)\dfn\{y\in X\st d(f^{-n}(x),f^{-n}(y))\lto{n\to+\infty}0\},
%\\W^{cu}(x)&\dfn\{y\in X\st\{d(f^{-n}(x),f^{-n}(y))\}_{n\in\N}\text{ is bounded}\},
%\end{aligned}
\]%\end{equation}
to get
\begin{theorem}\Label{PRPHopfErgodic}
If\/ $(X,\mu)$ is a metric Borel probability space, $f\colon X\to X$
invertible $\mu$-preserving, then any $f$-invariant $\varphi\in L^2(\mu)$
is $W^{ss}$- and $W^{su}$-saturated.
\end{theorem}
\begin{definition}%\Label{DEF}
Let $f\colon X\to X$ be a Borel-measurable map of a metric space $X$. An
$f$-invariant Borel probability measure $\mu$ is said to be ergodic (or $f$
to be \emph{ergodic} with respect to $\mu$) if every $f$-invariant
measurable set is either a null set or the complement of one. Equivalently,
every bounded measurable $f$-invariant function $\varphi$ is constant a.e.:
$\varphi\circ f=\varphi\Rightarrow\varphi\aeq\const$
\end{definition}
By analogy and as the link between \Ref{PRPHopfErgodic} and ergodicity we define
\begin{definition}%\Label{DEF}
$W^{ss}$, $W^{su}$ are said to be \emph{jointly ergodic} if 
\[%begin{equation}\label{eqsuffforergodic}
\varphi\in L^2(\mu)
,\ W^{ss}\text{-saturated and }W^{su}\text{-saturated}\Rightarrow\varphi\aeq\const
\]%end{equation}
\end{definition}
\begin{theorem}\Label{CORHopfErgodic}
If\/ $(X,\mu)$ is a metric Borel probability space, $f\colon X\to X$
invertible $\mu$-preserving, $W^{ss}, W^{su}$ \emph{jointly ergodic},
then $f$ is ergodic.
\end{theorem}
Although this paper gives a substantial strengthening of this classical
conclusion, we note a well-known simple one that is not often made
explicit. Since joint ergodicity is unaffected if we replace $f$ by $f^n$,
we actually have
\begin{theorem}\Label{THMHopfErgodic}
If\/ $(X,\mu)$ is a metric Borel probability space, $f\colon X\to X$
invertible $\mu$-preserving, $W^{ss}, W^{su}$ \emph{jointly ergodic}, then
$f$ is totally ergodic.
\end{theorem}
Here
\begin{definition}%\Label{DEF}
$f$ is said to be \emph{totally ergodic} if $f^n$ is ergodic for all
  $n\in\N^*$.
\end{definition}
\begin{remark}%\Label{REM}
This is equivalent to having no roots of unity in the spectrum of the
associated Koopman operator on $L^2$ and to having no adding machine or
permutation on a finite set as a factor \cite[p.\ 119]{Glasner}. We can, of
course, conclude in \Ref{THMHopfErgodic} that $f^n$ is ergodic for
$n\in\Z\smallsetminus\{0\}$.
\end{remark}
We can now state more explicitly the objectives of this paper. 
\begin{itemize}
\item Explain how joint ergodicity of the partitions implies more than total
ergodicity of $f$, namely mixing (\Ref{CORHopfCoudeneMixing}).
\item Give a nontrivial application (to partially hyperbolic dynamical
  systems, \Ref{THMBurnsWilkinson}).
\item Explain how the stronger assumption of ergodicity of $W^{ss}$ (alone)
  implies even more, namely multiple mixing (\Ref{CORWsergodicimpliesmixing}).
\item Establish criteria for ergodicity of $W^{ss}$ (\Ref{THMAbstractMMix}).
\item Give nontrivial applications (\eg to billiards, \Ref{CORBilliardsMMixing}).
\end{itemize}
%The point of \Ref{CORHopfCoudeneMixing} is that whenever this reasoning can
%be brought to bear, one can obtain mixing rather than ergodicity. We
%further show that in remarkable generality one can then improve this to
%multiple mixing.
The approach that improves ergodicity to mixing and multiple mixing gives
rise to a question which we make explicit here by way of previewing the
approach.

For establishing ergodicity, there is the trivial step from
\Ref{THMHopfCoudene} to \Ref{PRPHopfErgodic} ($f$-invariant functions are
$f^{-1}$-invariant). For establishing mixing, this is echoed below in the
nontrivial step from \Ref{THMHopfCoudeneMixing} for $N=1$ to
\Ref{THMHopfCoudeneMixing2sided} (weak accumulation points of\/
$\varphi\circ f^n$ are $W^{ss}$- \emph{and} $W^{su}$-saturated) which is
originally due to Babillot. We have no corresponding step for establishing
multiple mixing from joint ergodicity of $W^{ss}$ and $W^{su}$, that is, we
do not know how to go from \Ref{THMHopfCoudeneMixing} for $N>1$ to a
corresponding statement about $W^{ss}$- \emph{and} $W^{su}$-saturation.
\begin{question}%\Label{REM}
If\/ $X$ is a metric space, $f\colon X\to X$, $\mu$ an $f$-invariant Borel
probability measure, $\varphi_i\in L^2(\mu)$, then is any weak accumulation
point of\/ $\prod_{i=1}^N\varphi_i\circ f^{\sum_{j=1}^in_j}$ with
$n_i\lto{n\to\infty}\infty$ $W^{ss}$-saturated \emph{and}
$W^{su}$-saturated?
\end{question}
An affirmative answer would say that in our results that conclude
``mixing'' one does, in fact, have multiple mixing.
\subsection{Examples}
\begin{example}\Label{EXNotMixing}
Transformations of the form $f\times(-1)\colon X\times\{1,-1\}\to
X\times\{1,-1\}$, $(x,y)\mapsto(f(x),-y)$ are not mixing regardless of the
ergodic properties of\/ $f$.
%this illustrates the step between the 2 parts of \Ref{THMNMixingCriterion}. 
While in this case the finitary reduction given
by the return map to $X\times\{1\}$ may produce a mixing transformation,
the corresponding counterpart for flows is a suspension, in which the
absence of mixing is deemed substantial.
\end{example}
While the dynamical systems in which we are interested are
differentiable---either diffeomorphisms or flows---our interest is in the
ergodicity and related properties of Borel probability measures invariant
under the dynamical system. In the mechanical (that is, Hamiltonian) case,
this would, for instance be the so-called Liouville volume. We mentioned
geodesic flows as the original motivating examples, and we now add others
to our discussion. For all of these we will prove multiple mixing 
via the Hopf argument, that is, without recourse to sophisticated results
from entropy theory and the theory of measurable partitions in the context
of hyperbolic dynamical systems.
\begin{example}\Label{EX2111}
The action of\/ $\left(\begin{smallmatrix}2&1\\1&1\end{smallmatrix}\right)$ on $\R^2$ projects to an area-preserving diffeomorphism
$F_{\left(\begin{smallmatrix}2&1\\1&1\end{smallmatrix}\right)}\colon\T^2=\R^2/\Z^2\to\T^2$.
Distances on lines parallel to the eigenline $y=\dfrac{\sqrt{5}-1}{2}x$ for
the eigenvalue $\lambda_1=\dfrac{3+\sqrt{5}}{2}>1$ are expanded by a factor
$\lambda_1$. Similarly, the lines $y=\dfrac{-\sqrt{5}-1}{2}x+\const$
contract by $\lambda_1^{-1}=\lambda_2=\dfrac{3-\sqrt{5}}{2}<1$. 
%This expansion and contraction define hyperbolicity.
\end{example}
\begin{example}\Label{EXGenHypAutomorphism}
More generally, any $A\in GL(m,\Z)$ induces an automorphism $F_A$ of\/
$\T^m$ that preserves Lebesgue measure. We say that it is hyperbolic if\/ $A$
has no eigenvalues on the unit circle.
\end{example}
\begin{example}[{\cite[p.\ 104]{Walters2},\cite[p.\ 49]{Walters}}]\Label{Waltersautomorphism}
Likewise, $W\dfn
\left(\begin{smallmatrix}
0&0&0&-1\\
1&0&0&\phantom-8\\
0&1&0&-6\\
0&0&1&\phantom-8
\end{smallmatrix}\right)$ induces a volume-preserving automorphism $F_W$ of\/ $\T^4$.
The eigenvalues $2-\sqrt3\pm i\sqrt{4\sqrt3-6}$ lie on the unit circle and
the eigenvalues $\lambda_\pm=2+\sqrt3\pm\sqrt{2(3+2 \sqrt3)}\in\R$
satisfy $0<\lambda_-<1<\lambda_+$. $F_W$ is thus
\emph{partially hyperbolic}. The components of the
eigenvectors
\[
v^\pm\dfn(-2-\sqrt3\pm\sqrt{2(3+2\sqrt3)},3\mp2\sqrt{2(-3+2\sqrt3)},-6+\sqrt3\pm\sqrt{2(3+2\sqrt3)},1)
\]
are independent over $\Q$, \ie generate a 4-dimensional vector space over $\Q$.
%generated by $-2-\sqrt3+\sqrt{2(3+2\sqrt3)}$, $3-2\sqrt{2(-3+2\sqrt3)}$,
%$-6+\sqrt3+\sqrt{2(3+2\sqrt3)}$, and 1 is 4-dimensional.
\end{example}
\begin{example}[{\cite[p.\ 67]{ChernovMakarian}}]\Label{EXDispersing}
A billiard $\Dc\subsetneq\T^2$ is said to be \emph{dispersing} if it is
defined by reflection in the boundary of smooth strictly convex
``scatterers.''\footnote{One can allow corners at considerable expense of
  additional effort \cite[p.\ 69]{ChernovMakarian}.} If it has no corners
or cusps, then Sinai's Fundamental Theorem of the theory of dispersing
billiards \cite{BunimovichSinai,Sinai}, see also \cite[Theorem
  5.70]{ChernovMakarian}, establishes hyperbolic behavior of the billiard
map.
\end{example}
\begin{example}\Label{EXPolygonalPockets}
Sinai's Fundamental Theorem also applies to \emph{polygonal billiards with
  pockets}. These are noncircular billiards obtained from a convex polygon
as follows: for each vertex add a disk whose interior contains this vertex
and none other \cite[Theorem 4.1]{ChernovTroubetzkoy}.
\end{example}
\begin{example}\Label{EXKatokMap}
The \emph{Katok map} is a totally ergodic area-preserving deformation of
$F_{\left(\begin{smallmatrix}2&1\\1&1\end{smallmatrix}\right)}$ that is on
the boundary of the set of Anosov diffeomorphisms (hence not uniformly
hyperbolic) and whose stable and unstable partitions are homeomorphic to
those of $F_{\left(\begin{smallmatrix}2&1\\1&1\end{smallmatrix}\right)}$
\cite[\S1.3]{BarreiraPesin13}, \cite[\S6.3]{BarreiraPesin07},
\cite[\S2.2]{Katok}, \cite{PesinSentiZhang}.
\end{example}
\subsection{Ergodicity and related notions}
Since the \emph{time-averages} or \emph{Birkhoff averages}
$\frac1n\sum_{i=0}^{n-1}\varphi\circ f^i$ converge a.e. (Birk\-hoff
Pointwise Ergodic Theorem) and in $L^2$ (von Neumann Mean Ergodic Theorem),
ergodicity is equivalent to time averages coinciding with space averages
($\int\varphi$); this conclusion was the actual object of the
Maxwell--Boltzmann Ergodic Hypothesis. The motivation is that such
functions $\varphi$ represent \emph{observables} by associating to each
state of the system (each point in the domain of the dynamical system) a
number that might be the result of an experimental measurement. We note
that in this context we can use all\/ $L^p$ spaces ($p\in[1,\infty]$)
interchangeably: for any $p\in[1,\infty]$ ergodicity of $f$ is equivalent
to $f$-invariant $L^p$ functions being constant.

A simple nontrivial example of an ergodic transformation is $x\mapsto
x+\alpha\pmod1$ on $S^1=\R/\Z$ for irrational\/ $\alpha$ (Kronecker--Weyl
Equidistribution Theorem \cite[Proposition 4.2.1]{KatokHasselblatt}). The
preceding examples are also ergodic (with respect to the area measure), but
unlike an irrational circle rotation, they have stronger stochastic
properties, and the aim of this note is to show that the Hopf argument
yields them.

A colloquial motivation for these is that if\/ $\varphi$ represents the sugar
concentration in a cup with a lump of sugar, then rotation of the cup does
little to mix (and dissolve) the sugar.
\begin{definition}%\Label{DEF}
An $f$-invariant Borel probability measure is said to be \emph{mixing} if
two observables become asymptotically independent or uncorrelated when
viewed as random variables: 
\begin{equation}\label{eqMixing}
\int\varphi\circ f^n\;\psi\lto{n\to\infty}\int\varphi\int\psi\quad\text{for
all}\quad\varphi,\psi\in L^2.
\end{equation}
Equivalently, 
\begin{equation}\label{eqMixingWeak}
\varphi\circ f^n\xrightarrow[n\to\infty]{\text{ \tiny\normalfont weakly }}\const\quad\text{for all}\quad\varphi\in L^2.
\end{equation}
\end{definition}
With test function $\psi\equiv1$ in \eqref{eqMixing}, the left-hand side is
independent of\/ $n$, which shows that the constant on the right-hand side of
\eqref{eqMixingWeak} is $\int\varphi$.
\begin{definition}%\Label{DEF}
$\mu$ is said to be \emph{multiply mixing} if it is \emph{$N$-mixing} for all\/ $N\in\N$:
For $\varphi_1,\dots,\varphi_N\in L^\infty$ and any $L^2$-weak neighborhood
$U$ of (the constant function) $\prod_{i=1}^N\int\varphi_i\,d\mu$ there is
a $K\in\R$ such that $\prod_{i=1}^N\varphi_i\circ f^{\sum_{j=1}^in_j}\in U$
whenever $n_i\ge K$ for $1\le i\le N$. In short,
%\[
%\prod_{i=1}^N\varphi_i\circ f^{n_i}\xrightarrow[n_i-n_{i-1}\to\infty]{\text{$L^2$-weakly }}\prod_{i=1}^N\int\varphi_i\,d\mu\quad\text{for}\quad\varphi_i\in L^\infty.
%\tag{with $n_0\dfn0$}\]
\[
\prod_{i=1}^N\varphi_i\circ f^{\sum_{j=1}^in_j}\xrightarrow[n_i\to\infty]{\text{$L^2$-weakly }}\prod_{i=1}^N\int\varphi_i\,d\mu\quad\text{for}\quad\varphi_i\in L^\infty.
\]
\end{definition}
Made explicit with test function $\varphi_0$, this means that $N+1$
observables become asymptotically independent as the time gaps between them
go to infinity. Here, the left-hand side is
parametrized by $\Z^N$, and the assertion can be checked by considering
sequences $\psi_k=\prod_{i=1}^N\varphi_i\circ f^{\sum_{j=1}^in_j(k)}$ with $n_i(k)\lto{k\to\infty}\infty$ and
$\psi_k\xrightarrow[k\to\infty]{\text{ \tiny\normalfont weakly }}\psi$; then $\psi$ is an
accumulation point, and we describe these as ``weak accumulation points
$\psi_k\xrightarrow[k\to\infty]{\text{ \tiny\normalfont weakly }}\psi$ of\/
$\prod_{i=1}^N\varphi_i\circ f^{\sum_{j=1}^in_j(k)}$ with
$n_i(k)\lto{k\to\infty}\infty$'' or as ``weak accumulation points of\/
$\prod_{i=1}^N\varphi_i\circ f^{\sum_{j=1}^in_j}$ as
$n_i\to\infty$.'' $N$-mixing means that for
$\varphi_i\in L^\infty$ there is only one weak accumulation point
%$\psi_n\xrightarrow[n\to\infty]{\text{ \tiny\normalfont weakly }}\psi$ 
of\/
$\prod_{i=1}^N\varphi_i\circ f^{\sum_{j=1}^in_j}$ with
$n_i\to\infty$, and it is
$\prod_{i=1}^N\int\varphi_i\,d\mu$.
\begin{proposition}\Label{PRPMultiMixLemma}
An $f$-invariant Borel probability measure $\mu$ is $N$-mixing iff given any $\varphi_i\in L^2(\mu)$, any weak accumulation point
%$\psi_n\xrightarrow[n\to\infty]{\text{ \tiny\normalfont weakly }}\psi$ 
of\/
$\prod_{i=1}^N\varphi_i\circ f^{\sum_{j=1}^in_j}$ with
$n_i\to\infty$ is constant.
\end{proposition}
\begin{proof}
``Only if'' is clear. To prove ``if'', we recursively determine the
constant.

First, take $\varphi_i\equiv1$ for $i\neq1$, including taking the test
function $\varphi_0\equiv1$. Then the weak-accumulation statement becomes
\[
\int\varphi_1=\int\varphi_1\circ f^{n_1}\cdot1\to\const\int1=\const,
\]
so the constant is $\int\varphi_1$ for each such subsequence, and thus
$\varphi_1\circ f^{n_1}\xrightarrow[n_1\to\infty]{\text{ weakly
}}\int\varphi_1$. By symmetry, $\varphi_i\circ f^{n_i}\xrightarrow[n_i\to\infty]{\text{ weakly
}}\int\varphi_i$ for all\/ $i$. 
%In particular, $\varphi_2\circ f^{n_2-n_1}\xrightarrow[n_2-n_1\to\infty]{\text{ weakly }}\int\varphi_2$.
Next, if $\varphi_i\equiv1$ for $i\notin\{1,2\}$, then
\[%\begin{gather*}
\int\varphi_1\circ f^{n_1}\cdot\varphi_2\circ f^{n_1+n_2}\cdot1=\int\varphi_2\circ f^{n_2}\cdot\varphi_1\lto{n_2\to\infty}\int\varphi_1\int\varphi_2
\]
by the first step, so
\[\varphi_1\circ f^{n_1}\cdot\varphi_2\circ f^{n_1+n_2}\xrightarrow[n_1,n_2\to\infty]{\text{ weakly
}}\int\varphi_1\int\varphi_2
\]%end{gather*}
with like statements for any pair of the $\varphi_i$. This can be
continued, and the existence of an accumulation point (by the
Banach--Alaoglu Theorem) completes the proof.
%
% the left-hand side of the weak-convergence statement (with test
%function $\varphi_0\equiv1$) becomes
%\[
%\int\varphi_2\circ f^{n_2}\varphi_1\circ f^{n_1}\cdot1=\int\varphi_2\circ f^{n_2-n_1}\varphi_1,
%\]
%and since $n_2-n_1\to\infty$,
%
%
%functions $\varphi$ and $\psi$ all weak accumulation points of\/ $\varphi\circ
%f^n\cdot\psi\circ f^m$ are constant. By taking 1 as test function, this
%constant is the corresponding accumulation point of
%\[\int\varphi\circ f^n\cdot\psi\circ f^m=\int\varphi\circ f^{n-m}\psi\lto{n-m\to\infty}\int\varphi\int\psi\]
%by the statement for $N=1$. This can be continued inductively.
\end{proof}
\section{The one-sided Hopf Argument yields multiple mixing}\Label{SHopfArgument}
%We now present what we call the one-sided Hopf Argument.
We note that the following uses no compactness or exponential contraction.
\begin{proposition}[{\cite[\S 3.3]{CoudeneHDR}}]\Label{THMHopfCoudeneMixing}
If\/ $(X,\mu)$ is a metric Borel probability space, $f\colon X\to X$
$\mu$-preserving, $\varphi_i\in L^2(\mu)$, then weak accumulation points
%$\psi_n\xrightarrow[n\to\infty]{\text{ weakly }}\psi$ 
of\/ $\prod_{i=1}^N\varphi_i\circ f^{\sum_{j=1}^in_j}$ with
$n_i\lto{n\to\infty}\infty$ are $W^{ss}$-saturated.
\end{proposition}
\Ref{PRPMultiMixLemma} gives a strong immediate consequence of
\Ref{THMHopfCoudeneMixing}:
\begin{theorem}\Label{CORWsergodicimpliesmixing}
$f$ is multiply mixing if every $W^{ss}$-saturated $\varphi\in L^2$ is constant a.e.
\end{theorem}
\begin{proof}[Proof of \Ref{THMHopfCoudeneMixing}]
By the Banach--Saks Lemma $\psi_n\xrightarrow[n\to\infty]{L^2\text{-weakly
}}\psi$ has a subsequence for which
$\displaystyle\frac1n\sum_{k=0}^{n-1}\psi_{n_k}\xrightarrow[n\to\infty]{\ L^2\ }\psi$.
%Note that we have given up a little by passing to a subsequence and to a
%Birkhoff average rather than a limit but gained by achieving
%$L^2$-convergence rather than weak convergence. 
Furthermore,
$\psi_n\xrightarrow[n\to\infty]{\ L^2\ }\psi$ implies that there is a
subsequence with $\psi_{n_k}\xrightarrow[k\to\infty]{\text{ a.e.\ }}\psi$. 
%Here we again give up a little by passing to a subsequence but ``upgrade''
%to pointwise convergence. 
This gives subsequences $m_l$, $n_{i_k}$
with
\[
\Psi_l\dfn\frac1{m_l}\sum_{k=0}^{m_l-1}\psi_{n_{i_k}}\xrightarrow[l\to\infty]{\text{ a.e.\ }}\psi.
\]
Pointwise convergence makes this $W^{ss}$-saturated for bounded uniformly
continuous functions:  $p^l_{ij}\dfn\varphi_i(f^{(n_i)_l}(x_j))$ for
$j=1,2$ with $x_2\in W^{ss}(x_1)$ gives
\[
\prod_{i=1}^Np^l_{i2}-\prod_{i=1}^Np^l_{i1}
=
\sum_{\ell=1}^N
\big[\prod_{i=1}^{\ell-1}p^l_{i2}\big]
\big[p^l_{\ell2}-p^l_{\ell1}\big]
\big[\prod_{i=\ell+1}^Np^l_{i1}\big]
\lto{l\to\infty}0.
\]
Approximate $\varphi_i^0\in L^\infty\cap L^2$ within $1/k$ by bounded
uniformly continuous $\varphi_i^k$ and let $p^l_{ij}\dfn\varphi_i^j\circ
f^{(n_i)_l}$. Then weak limits (of subsequences if necessary) satisfy
\[
\!\!\|\psi-\psi^k\|\le\varliminf_{l\to\infty}\big\lVert\prod_{i=1}^Np^l_{ik}-\prod_{i=1}^Np^l_{i0}\big\rVert
\le
\sum_{\ell=1}^N
\prod_{i=1}^{\!\!\ell-1\!\!}\big\lVert p^l_{i2}\big\rVert_\infty
\big\lVert p^l_{\ell2}-p^l_{\ell1}\big\rVert_2
\prod_{\!\!\!\!i=\ell+1\!\!\!\!}^N\big\lVert p^l_{i1}\big\rVert_\infty
\lto{k\to\infty}0
\]
so, after passing to a subsequence, $\psi^k\aeto\psi$, which is hence $W^{ss}$-saturated.
\end{proof}
In \Ref{EX2111} the contracting lines have irrational slope, so the
intersections of each with the circle $S^1\times\{0\}\subset S^1\times
S^1=\T^2$ are the orbit of an irrational rotation---whose ergodicity
implies that the stable partition $W^{ss}$ is ergodic \cite[Proposition
  4.2.2]{KatokHasselblatt}. The ``one-sided''
\Ref{CORWsergodicimpliesmixing} gives
\begin{proposition}%\Label{PRP}
$F_{\left(\begin{smallmatrix}2&1\\1&1\end{smallmatrix}\right)}$ is multiply
mixing with respect
to Lebesgue measure.
\end{proposition}
\begin{remark}%\Label{REM}
Simple Fourier analysis also establishes this conclusion, but
while linearity is helpful for the Hopf argument, it is indispensable for
Fourier analysis. 
\end{remark}
\begin{remark}%\Label{REM}
Instead of ergodicity of an irrational rotation, one can use
\Ref{THMAbstractMMix}, and it may be of
interest to read the proofs with
$\left(\begin{smallmatrix}2&1\\1&1\end{smallmatrix}\right)$ in mind.\looseness-1
\end{remark}
\begin{remark}%\Label{REM}
  A volume-preserving $C^1$ perturbation of\/ $F_{\left(\begin{smallmatrix}2&1\\1&1\end{smallmatrix}\right)}$ is a
topologically conjugate Anosov diffeomorphism for which the local product
charts can be chosen to be differentiable. (More generally, any
volume-preserving Anosov diffeomorphism of $\T^2$ is topologically
conjugate to a hyperbolic automorphism, and  the local product
charts can be chosen to be differentiable.) Thus, we have a local product
structure and absolute continuity for free and obtain multiple mixing
from \Ref{THMAbstractMMix} and
\Ref{CORWsergodicimpliesmixing}.
\end{remark}
\begin{remark}\Label{REMsingular}
In contrast with Lebesgue measure, the measure that assigns 1/4 to each of
the points $\pm\nicefrac15\binom12$ and $\pm\nicefrac15\binom24$ has 2
ergodic components. The reader is encouraged to check where this affects
our proofs.
\end{remark}
The contracting lines of the partially hyperbolic automorphism in
\Ref{Waltersautomorphism} are generated by a vector whose components are
rationally independent, hence project to the orbits of an ergodic flow
\cite[p.\ 147]{KatokHasselblatt}. \Ref{CORWsergodicimpliesmixing} gives:
\begin{proposition}%\Label{PRP}
$F_W$ in \Ref{Waltersautomorphism} is multiply mixing.
\end{proposition}
\section{The two-sided Hopf argument yields mixing}\Label{SApplications}
The assumption in \Ref{CORWsergodicimpliesmixing} that $W^{ss}$ is ergodic
is rather strong, and the classical Hopf argument is based on joint
ergodicity of $W^{ss}$ and $W^{su}$. To improve this to mixing, we need to
augment the conclusion of \Ref{THMHopfCoudeneMixing} to include
$W^{su}$-saturation as well. This requires a slightly subtle argument.
\begin{theorem}[{\cite[Theorem 3]{Coudene}}]\Label{THMHopfCoudeneMixing2sided}
If\/ $(X,\mu)$ is a metric Borel probability space, $f\colon X\to X$
invertible $\mu$-preserving and $\varphi\in L^2(\mu)$, then any weak
accumulation point of\/ $U_f^n(\varphi)$ as $n\to+\infty$ is $W^{ss}$- and
$W^{su}$-saturated.
\end{theorem}
\begin{proof}[Proof \normalfont(\cite{Babillot,Coudene})]
Denote by $I\subset L^2(\mu)$ the (closed) subspace of functions
subordinate to $W^{ss}$ and $W^{su}$ and by $I^\perp\dfn\{\varphi\in
L^2\st\langle\varphi,\psi\rangle=0\text{ for }\psi\in I\}$ its
orthocomplement. To show $U_f^{n_i}(\varphi)\xrightarrow[i\to\infty]{\text{
weakly }}\psi\Rightarrow\psi\in I$ take
$\varphi=\varphi_I+\varphi^\perp\in I\oplus I^\perp=L^2$ and a subsequence
with $U_f^{n_{i_k}}(\varphi_I)\xrightarrow[k\to\infty]{\text{ weakly
}}\psi_I\in I$ and
$U_f^{n_{i_k}}(\varphi^\perp)\xrightarrow[k\to\infty]{\text{ weakly
}}\psi^\perp\perp I$. Then $\psi=\psi_I+\psi^\perp$, and we are done if we
find a $\psi'$ with
$\langle\psi^\perp,\psi^\perp\rangle=\langle\varphi^\perp,\psi'\rangle=0$.

By \Ref{THMHopfCoudeneMixing}, $\psi^\perp$ is subordinate to $W^{ss}$, and
hence so is any $U_f^{-n}(\psi^\perp)$ and any weak limit
$\psi'=\lim_{i\to\infty}U_f^{-{n_i}}(\psi^\perp)$, while
\Ref{THMHopfCoudeneMixing} applied to $\psi^\perp$ and $f^{-1}$ implies
that $\psi'$ is subordinate to $W^{su}$ as well, \ie $\psi'\in I$. Thus
\[
0=\langle\varphi^\perp,\psi'\rangle
=
\lim_{i\to\infty}\langle\varphi^\perp,U_f^{-{n_i}}(\psi^\perp)\rangle
=
\lim_{i\to\infty}\langle U_f^{n_i}(\varphi^\perp),\psi^\perp\rangle
=
\langle\psi^\perp,\psi^\perp\rangle.\qedhere
\]
\end{proof}
\Ref{THMHopfCoudeneMixing2sided} can alternatively be obtained from the
following result:
\begin{theorem}[{Derriennic--Downarowicz \cite[Th\'eor\`eme 2.4]{CoudeneEns}}]%\Label{THM}
A weak accumula\-tion point of $(U_f^n(\varphi))_{n\in\N}$ is a weak
accumulation point of $(U_f^{-n}(\phi))_{n\in\N}$ for some $\phi$\,\llap.\!\looseness-1
\end{theorem}
\Ref{THMHopfCoudeneMixing2sided} has the following consequences, as noted
in \cite{Coudene}:
\begin{theorem}\Label{CORHopfCoudeneMixing}
If\/ $(X,\mu)$ is a metric Borel probability space, $f\colon X\to X$
invertible $\mu$-preserving, $W^{ss}, W^{su}$ \emph{jointly ergodic}, then
$f$ is mixing.
\end{theorem}
\begin{theorem}\Label{CORHopfCoudeneErgodic}
If\/ $(X,\mu)$ is a metric Borel probability space, $f\colon X\to X$
invertible $\mu$-preserving, and
\[
\varphi\in L^2(\mu)\ f\text{-invariant}
,\ W^{ss}\text{-saturated and }W^{su}\text{-saturated}\Rightarrow\varphi\aeq\const,
\]
then $f$ is ergodic.
\end{theorem}
\begin{proof}
An $f$-invariant $\varphi$ is a weak accumulation point of\/ $\varphi=\prod_{i=1}^N\varphi_i\circ f^{n_i}$, hence $W^{ss}$- and
$W^{su}$-saturated by \Ref{THMHopfCoudeneMixing}, hence constant by
assumption.
\end{proof}
\Ref{THMHopfCoudeneMixing2sided} also holds for flows (mutatis mutandis),
and thus we get the following corollary:
\begin{corollary}\Label{CORHopfCoudeneflows}
Let $X$ be a metric space, $f^t\colon X\to X$ a flow, $\mu$ an
$f^t$-invariant Borel probability measure. If
\[
\varphi\in L^2(\mu)\ f^t\text{-invariant}
,\ W^{ss}\text{-saturated and }W^{su}\text{-saturated}\Rightarrow\varphi\aeq\const,
\]
then $f^t$ is ergodic, and joint ergodicity of $W^{ss},W^{su}$ implies that
$f^t$ is mixing.
\end{corollary}
Our aim is to obtain multiple mixing easily, but
\Ref{CORHopfCoudeneMixing} is interesting because of its weak
hypotheses. It applies where other methods do not \cite{Babillot}.
\section{Absolute continuity and product sets}
We now apply these results to the hyperbolic toral automorphisms of
\Ref{EXGenHypAutomorphism} to demonstrate the classical use of the Hopf
argument to get ergodicity, except that \Ref{CORHopfCoudeneMixing} yields
mixing instead.
\begin{proposition}\Label{PRPHopfargument2111}
If\/ $A\in GL(m,\Z)$ is hyperbolic, then the induced automorphism $F_A$ of\/
$\T^m$ is mixing with respect to Lebesgue measure.
\end{proposition}
\begin{proof}
For $q\in\T^m$ the \emph{stable} and \emph{unstable subspaces} at $q$ in
\eqref{eqDefStableSet} are
\[
W^{ss}(q)=\pi(E^-+q)\text{ and }W^{su}(q)=\pi(E^++q),
\]
where $E^\pm$ are the contracting and expanding subspaces of\/ $A$ and
$\pi\colon\R^m\to\T^m$ is the projection. Suppose $\varphi\in L^2$ is
$W^{ss}$- and $W^{su}$-saturated, \ie there is a conull\/ $G\subset\T^n$ such that
$x,y\in G$, $y\in W^{ss}(x)\Rightarrow\varphi(x)=\varphi(y)$ and $x,y\in G$,
$y\in W^{su}(x)\Rightarrow\varphi(x)=\varphi(y)$. We will prove that
$\varphi\aeq\const$, and \Ref{CORHopfCoudeneMixing} then implies mixing.

Let $D^\pm\subset E^\pm$ be small disks and $q\in \T^m$. Then $q$ has a
neighborhood that is up to rotation and translation of the form $D^-\times
D^+$, and
$C\dfn G\cap(D^-\times D^+)$
has full Lebesgue measure in $D^-\times D^+$, \ie if\/ $\mu^\pm$ denotes the
normalized Lebesgue measure on $D^\pm$ and $\mu=\mu^-\times\mu^+$, then
$\int_{D^-\times D^+}\chi_C\,d\mu=1$. By the Fubini Theorem
\[
1
=
\int_{D^-\times D^+}\!\!\chi_C\,d\mu
=
\int_{D^-}\int_{D^+}\!\!\chi_C\,d\mu^+\,d\mu^-,
\ \text{so}\ 
\int_{D^+}\!\!\chi_C(u,\cdot)\,d\mu^+=1\text{ for }\mu^-\text{-a.e.\ }u\in D^-.
\]
Fix such a $u_0\in D^-$, and note that by construction
$C^-\dfn D^-\times\big(C\cap(\{u_0\}\times D^+)\big)$
has full Lebesgue measure.\footnote{One might at this time revisit \Ref{REMsingular}.}
If\/ $(u,v),(u',v')\in C^-\cap C$, a set of full measure, then
\[
\varphi(u,v)
=
\varphi(u_0,v)
=
\varphi(u_0,v')
=
\varphi(u',v').
\]
This applies to any such neighborhood of an arbitrary $q\in\T^n$, so $\varphi\aeq\const$
\end{proof}
This is how Hopf established the ergodicity of geodesic flows of manifolds
of negative curvature. The method was extended to geodesic flows of
higher-dimensional manifolds by Anosov. The pertinent discrete-time
counterpart are Anosov diffeomorphisms, which include the $F_A$ above. As
the preceding argument shows, higher-dimensionality does not directly
affect the intrinsic difficulty of the argument. The barrier that Hopf
faced and Anosov overcame is related to the use of the Fubini Theorem
above---except in Hopf's context, where local product neighborhoods are
indeed diffeomorphic to euclidean patches, one needs to establish the
\emph{absolute continuity} of the invariant foliations on each such patch
to apply the Fubini Theorem (see, \eg\cite[Chapter 6]{BrinStuck}). This is
a natural point at which to define center-stable and -unstable sets 
\[%\begin{equation}%\label{eqDefWeakStableSet}
\begin{aligned}
W^{cs}(x)&\dfn\{y\in X\st\{d(f^n(x),f^n(y))\}_{n\in\N}\text{ is bounded}\},\\
W^{cu}(x)&\dfn\{y\in X\st\{d(f^{-n}(x),f^{-n}(y))\}_{n\in\N}\text{ is bounded}\}.
\end{aligned}
\]
\begin{definition}\Label{DEFAbsCtsLocPrStr}
Let\/ $(X,\mu)$ be a metric Borel probability space, $f\colon X\to X$
invertible $\mu$-preserving, $i\in\{ss,cs\}$, $j\in\{su,cu\}$.
We say that $V\subset X$ is an $(i,j)$-\emph{product set} if
for $x\in V$ and $k\in\{i,j\}$ there are $\Wloc^k(x)\subset W^k(x)$
and a measurable map $[\cdot,\cdot]\colon V\times V\to X$ with
%%%%\begin{equation}\label{eqLPS}
$[x,y]\in\Wloc^i(x)\cap\Wloc^j(y)$.
%\begin{align*}
%W^i_{r_i}(x)&\dfn\{y\in W^i(x)\st d(f^n(x),f^n(y))\le r_i\text{ for }r\in\N_0\},\\
%W^j_{r_j}(x)&\dfn\{y\in W^j(x)\st d(f^{-n}(x),f^{-n}(y))\le r_j\text{ for }r\in\N_0\}.
%\end{align*}
%In this case we denote by $[x,y]$
%the unique element of\/ $W^i_{r_i}(x)\cap
%W^j_{r_j}(y)$. 
%For $z\in V$, $[\cdot,\cdot]\colon W^u_{r_u}(z)\times W^s_{r_s}(z)\to V$ defines a \emph{product chart}.

We say that $W^i$ is \emph{absolutely continuous} on  an $(i,j)$-\emph{product set}  $V$ (with respect to
$\mu$) if for each $x\in V$ and $k\in\{i,j\}$  there are measures $\mu^k_x$ on $\Wloc^k(x)$ with
$\mu^j_x(N)=0\Rightarrow\mu^j_y([N,y])=0$ and $\phi\in L^1(\mu)\Rightarrow\int_V\phi
d\mu=\int_{\Wloc^i(z)}\int_{\Wloc^j(x)}\phi\;d\mu^j_x\;d\mu^i_z(x)$.% for $\phi\in L^1(\mu)$. 
%We say that $W^i$ has \emph{uniformly bounded
%Jacobians} if there is an $M\in\R$ such that
%$\mu^j_x(N)\le M\mu^j_y([N,y])$ for each such $(i,j)$-product
%set.
\end{definition}
Then one obtains \cite[Chapter 6]{BrinStuck}:
\begin{proposition}\Label{PRPAnosovMixing}
Volume-preserving Anosov diffeomorphisms are mixing.
\end{proposition}
\Ref{CORHopfCoudeneMixing} can be applied well beyond this completely
hyperbolic case. With the terminology of \cite{BurnsWilkinson} we have
\begin{theorem}\Label{THMBurnsWilkinson}
Let $f$ be $C^2$, volume-preserving, partially hyperbolic, and
center bunched. If $f$ is essentially accessible, then $f$ is mixing.
\end{theorem}
\begin{proof}
Every bi-essentially saturated set is essentially bisaturated
\cite[Corollary 5.2]{BurnsWilkinson}, so \Ref{CORHopfCoudeneMixing} applies
by essential accessibility \cite[p.\ 472]{BurnsWilkinson}.
\end{proof}
The main result of Burns and Wilkinson \cite[Theorem 0.1]{BurnsWilkinson}
is that $f$ is ergodic and in fact has the Kolmogorov property. They obtain
ergodicity from \cite[Corollary 5.2]{BurnsWilkinson} by the Hopf argument,
so by \Ref{CORHopfCoudeneMixing} one obtains mixing directly. The
Kolmogorov property is then obtained by invoking a result of Brin and Pesin
\cite{BrinPesin} that the Pinsker algebra is bi-essentially saturated in
this context. Our point is that here, too, the Hopf argument alone provides
mixing rather than just ergodicity without any ``high-tech'' ingredients.
\section{Applications: Multiple mixing}\Label{SThouvenot}
\Ref{CORHopfCoudeneMixing} says that $f$ is mixing if $\varphi\in
L^2(\mu)$, $W^{ss}$- and $W^{su}$-saturated $\Rightarrow\varphi\aeq\const$,
and in the previous section we established the ``if'' part of the
statement. Likewise, \Ref{CORWsergodicimpliesmixing} says that if every
$W^{ss}$-saturated $\varphi\in L^2$ is constant a.e., then $f$ is multiply
mixing, and we now (on page \pageref{eqsaturationgiveszerolimit}) verify
this ``if'' statement---in remarkable generality, such as in the original
context (of uniformly hyperbolic dynamical systems) in which the Hopf
argument applies in the manner shown in \Ref{SApplications}.
%We prove two results to show just how little is needed
%for ergodicity of\/ $W^s$.
%\subsection{Center-stable leaves}
The result does not use the contraction on\/ $W^{ss}$; thus it also applies
to $W^{cs}$, such as in \Ref{Waltersautomorphism} or to the
\emph{weak}-stable foliation of a flow.
\begin{theorem}\Label{THMAbstractMMix}
If\/ $(X,\mu)$ is a metric Borel probability space, $f\colon X\to X$
invertible $\mu$-preserving ergodic, $i\in\{ss,cs\}$. If\/ $W^i$ is
absolutely continuous on an $(i,su)$-product set $V$, and $\mu(f^{-1}(V)
\cap V)>0$,
%and
%\begin{equation}\label{equnstableexpansion}
%\sup_{x\in V}\diam f^{-n}(W^u_{r_u}(x))\lto{n\to\infty}0,
%\end{equation}
%where 
%$\diam(E)\dfn\sup\{d(x,y)\st x,y\in E\}$,
then $W^i$ is ergodic.
\end{theorem}
\begin{corollary}\Label{CORAbstractMMix}
If\/ $(X,\mu)$ is a metric Borel probability space, $f\colon X\to X$
invertible $\mu$-preserving totally ergodic, $W^{ss}$ absolutely continuous
on an $(ss,su)$-product set $V$ with $\mu(V)>0$. Then $f$ is multiply
mixing.
\end{corollary}
\begin{proof}
The Poincar\'e Recurrence Theorem gives an $N\in\N$ with $\mu(f^{-N}(V)\cap V)>0$. Apply \Ref{THMAbstractMMix} to
$f^N$, then \Ref{CORWsergodicimpliesmixing} to $f$.
\end{proof}
\Ref{THMHopfErgodic} makes it easy to establish total ergodicity. For
instance:
\begin{theorem}%\Label{THM}
Let $X$ be a separable metric space, $\mu$ a Borel probability measure with
connected support, $f\colon X\to X$ an invertible $\mu$-preserving
transformation. If $W^{ss}$ is absolutely continuous on open
$(ss,su)$-product sets that cover the support of $\mu$, then $f$ is totally
ergodic and thus multiply mixing by \Ref{CORAbstractMMix}.
\end{theorem}
\begin{proof}
Apply \Ref{THMHopfErgodic}: An $f$-invariant function is $W^{ss}$- and
$W^{su}$-saturated, hence by absolute continuity a.e.\ constant on these
product sets. A function on a connected set is a.e.\ constant if it is
a.e.\ locally constant.
\end{proof}
\begin{remark}%\Label{REM}
This applies to volume-preserving Anosov diff\-eo\-morph\-isms
\cite[Chapter 6]{BrinStuck} but we do not use exponential behavior,
differentiability or compactness.
\end{remark}
\begin{theorem}\Label{CORBilliardsMMixing}
The Liouville measure for dispersing billiards (\Ref{EXDispersing}) and for
polygonal billiards with pockets (\Ref{EXPolygonalPockets}) is multiply
mixing.
\end{theorem}
\begin{proof}
For dispersing billiards, Sinai's Fundamental Theorem of the theory of
dispersing billiards \cite[Theorem 5.70]{ChernovMakarian} provides product
sets \cite[Proposition 7.81]{ChernovMakarian} with absolutely continuous
holonomies \cite[Theorem 5.42]{ChernovMakarian}, which implies the absolute
continuity property we use. \Ref{CORHopfCoudeneMixing} then establishes
mixing and hence total ergodicity, which by \Ref{CORAbstractMMix} implies
mixing of all orders. This also works for polygonal billiards with pockets
\cite[Theorem 4.1]{ChernovTroubetzkoy}.
\end{proof}
\begin{theorem}%\Label{THM}
The Katok map (\Ref{EXKatokMap}) is multiply mixing.
\end{theorem}
\begin{proof}
It is totally ergodic and the stable and unstable partitions are
homeomorphic to those of
$F_{\left(\begin{smallmatrix}2&1\\1&1\end{smallmatrix}\right)}$
(\Ref{EXKatokMap}), so there is a product neighborhood, which hence has
positive measure. Absolute continuity on this neighborhood follows from
Pesin theory, so we can apply \Ref{CORAbstractMMix}.
\end{proof}
In fact, the proofs of \Ref{THMAbstractMMix} and \Ref{CORAbstractMMix}
applied to the pieces of the ergodic decomposition of\/ $\mu$ yield:
\begin{corollary}%\Label{COR}
Let $X$ be a metric space, $\mu$ a Borel probability measure, $f\colon X\to
X$ $\mu$-preserving invertible, $i\in\{ss,cs\}$ such that
every point is in an $(i,su)$-product set where $W^i$ is absolutely continuous. If\/
$\varphi\colon X\to\R$ is $W^i$-saturated, then there is a measurable
$f$-invariant $n\colon X \to \N$ with $\varphi(f^{n(x)}(x)) = \varphi(x)$
a.e.
\end{corollary}
\begin{lemma}%\Label{LEM}
Absolute continuity of\/ $W^i$ on $V_f\dfn f^{-1}(V) \cap V$ implies absolute
continuity of\/ $T\colon V_f\to X$, $x\mapsto
T(x)\dfn[f(x),x]$, \ie $T_*\mu \ll \mu$.
\end{lemma}
\begin{proof}
If\/ $N\subset V_f$ and $\mu(N)=0$, then there is a $\Wloc^{su}$-saturated null set
$N_W$ such that for $z\notin N_W$ we have $\int_{W^{su}(z)}\chi_N\;d\mu^{su}_z=0$
as well as, by $f$-invariance of\/ $\mu$ and absolute continuity,
$\int_{W^{su}(z)}\chi_{T(N)}\;d\mu^{su}_z=0$. Then
%\begin{setlength}{\multlinegap}{0pt}
\begin{align*}
\int\chi_{T(N)}\;d\mu
&=
\int_{\Wloc^i(z)\smallsetminus N_W}\int_{\Wloc^{su}(x)}\chi_{T(N)}\;d\mu^{su}_x\;d\mu^i_z(x)
%\\&\qquad
+
\int_{N_W}\chi_{T(N)}\;d\mu
\\&=
\int_{\Wloc^i(z)\smallsetminus N_W}0\;d\mu^i_z(x)
+
0
.\qedhere\end{align*}
%\end{setlength}
\end{proof}
We adapt an idea of Thouvenot \cite[Theorem 1]{Thouvenot}, \cite[Exercice
  7, p.\ 50]{CoudeneBook}, \cite[Proposition 1.2]{CoudeneEns}: $d(
f^{-n}(x), f^{-n}(T(x)) )\to0$ pointwise on $ V_f\cap T^{-1}V_f$, hence by
the Egorov theorem uniformly on some $U\subset V_f\cap T^{-1}V_f$ with
$\mu(U)>0$. Then $T_n\dfn\begin{cases}f^{-n} \circ T \circ f^{n}&\text{on
}f^{-n}(U)\\\Id&\text{elsewhere}\end{cases}\xrightarrow[n\to\infty]{\text{\Tiny
    pointwise}}\Id$, and $T_n$ has Radon--Nikodym derivative
$g_n\dfn\Big[\frac{d {T_n}_*
    \mu}{d\mu}\Big]=\Big[\frac{dT_*\mu}{d\mu}\Big]\circ f^{n}$ on
$f^{-n}(U)$ (and 1 elsewhere); this is uniformly integrable, \ie
$\sup_{n\in\N}\int_{\{g_n > M\}} g_n d\mu\lto{M\to\infty}0$.
\begin{lemma}\Label{LEMAbstractMMix}
Let $X$ be a metric space with probability measure $\mu$, $T_n\colon X\to
X$ such that $T_n\to\Id$ a.e., ${T_n}_* \mu \ll \mu$, and $g_n\dfn\Big[\frac{d {T_n}_*
\mu}{d\mu}\Big]$ is uniformly integrable. Then $\|\varphi
\circ T_n-\varphi\|_1\lto{n\to\infty}0$ for all\/ $\varphi \in L^\infty$.
($\|\cdot\|_p$ denotes the $L^p$-norm.)
\end{lemma}
\begin{proof}
If\/ $\psi$ is continuous with $\|\psi\|_\infty\le\|\varphi\|_\infty$, then
\[
\|\varphi \circ T_n - \varphi\|_1 \leq
\|(\varphi - \psi) \circ T_n\|_1 + \|\psi \circ T_n -
\psi\|_1+ 
\|\psi-\varphi\|_1,
\]
and $\|\psi\circ T_n-\psi\|_1\lto{n\to\infty}0$ by the Bounded Convergence Theorem.
For $\epsilon>0$, uniform integrability provides an $M$ such that  the last
summand in
\[
\|(\varphi - \psi) \circ T_n\|_1 = \int |\psi-\varphi|_1 g_n d\mu
\leq
M \|\psi-\varphi\|_1 +
2\|\varphi\|_\infty  \int_{\{g_n > M\}} g_n d\mu
\]
is less than $\epsilon/2$. Choose $\psi$ such that
$\|\psi-\varphi\|_1<\frac{\epsilon/2}{M+1}$.
\end{proof}
\begin{proof}[Proof of \Ref{THMAbstractMMix}]
%For $\varphi$, \Ref{LEMAbstractMMix} implies
%\[
%\mu(\big\{|\varphi \circ T_n - \varphi| > \varepsilon\big\})\lto{n\to\infty}0.
%\]
Let\/ $\varphi\in L^\infty$ be $W^i$-saturated. We show $\varphi$ is
$f$-invariant. If $\varepsilon>0$, then $T_n(x) \in W^i(f(x))$ for all\/ $x\in
f^{-n}(U)$ implies that
\begin{equation}\label{eqsaturationgiveszerolimit}
\mu\Bigl(f^{-n}(U) \cap\big\{|\varphi\circ f - \varphi|> \varepsilon\big\}\Bigr)
= 
\mu\Bigl(f^{-n}(U) \cap\big\{|\varphi \circ T_n - \varphi| > \varepsilon\big\}\Bigr)
\xrightarrow[n\to\infty]{\text{(\Ref{LEMAbstractMMix})}}0.
\end{equation}
With $B\dfn\big\{|\varphi\circ f-\varphi|>\varepsilon\big\}$, the Mean
Ergodic Theorem and ergodicity of\/ $f$ imply
\[
\frac{1}{n}\sum^{n-1}_{k=0}\chi_U\circ f^k\xrightarrow[n\to\infty]{L^2}\mu(U),
\quad\text{
hence
}\quad
\frac{1}{n}\sum^{n-1}_{k=0}\chi_U\circ f^k\chi_B\xrightarrow[n\to\infty]{L^2}\mu(U)\chi_B,
\]
so
$0\xleftarrow[n\to\infty]{\eqref{eqsaturationgiveszerolimit}}\frac{1}{n}\sum^{n-1}_{k=0}\mu\Bigl(f^{-n}(U)\cap\big\{|\varphi\circ
T_n-\varphi|>\varepsilon\big\}\Bigr)\lto{n\to\infty}\mu(U)\,\mu(B)$. Since
$\mu(U)>0$, we have $\mu(B)=0$. $\epsilon$ was arbitrary, so $\varphi$ is
$f$-invariant, hence constant a.e.
\end{proof}

\end{document}